\documentclass[12pt]{amsart}

\usepackage{amsmath}
\usepackage{stmaryrd}
\usepackage{amsfonts}
\usepackage{amssymb}
\usepackage{amsthm}
\usepackage{hyperref}
\usepackage{graphicx}
\usepackage{enumerate}
\usepackage{cleveref}
\usepackage{mathrsfs}
\usepackage{geometry}

\newtheorem{theorem}{Theorem}[section]
\newtheorem*{thm*}{Theorem}
\numberwithin{equation}{section}
\newtheorem{lemma}[theorem]{Lemma}
\newtheorem{proposition}[theorem]{Proposition}
\newtheorem{prop}[theorem]{Proposition}
\newtheorem{corollary}[theorem]{Corollary}
\newtheorem{cor}[theorem]{Corollary}
\newtheorem{definition}[theorem]{Definition}
\newtheorem{conjecture}[theorem]{Conjecture}
\Crefname{conjecture}{Conjecture}{Conjectures}

\theoremstyle{remark}

\newtheorem*{remark*}{Remark}
\newtheorem*{rem}{Remark}
\newtheorem{example}[theorem]{Example}
\theoremstyle{plain}

\newcommand{\N}{\mathbb{N}}

\newcommand{\Z}{\mathbb{Z}}
\newcommand{\Q}{\mathbb{Q}}
\newcommand{\cO}{{\mathcal O}} 
\newcommand{\R}{\mathbb{R}}
\newcommand{\C}{\mathbb{C}}

\newcommand{\eps}{\varepsilon}

\newcommand{\SL}{\operatorname{SL}}

\newcommand{\calH}{\mathscr{H}}

\renewcommand{\Im}{\operatorname{Im}}

\newcommand{\SLZ}{\SL_2(\Z)}
\newcommand{\abcd}{\left(\begin{smallmatrix} a & b \\ c & d \end{smallmatrix}\right)}

\newcommand{\HH}{\mathbb{H}}

\newcommand{\z}{\mathfrak{z}}

\newcommand{\sfd}{\mathsf{d}}
\newcommand{\tM}{\widetilde{M}}

\newcommand{\Nm}{\operatorname{Nm}}

\author{Pavel Guerzhoy, Michael H. Mertens, and Larry Rolen}
\address{Department of Mathematics,
University of Hawaii,
2565 McCarthy Mall
Honolulu, HI 96822 }
\email{pavel@math.hawaii.edu}
\address{Max Planck Institute for Mathematics,
Vivatsgasse 7,
53111 Bonn,
Germany}
\email{mhmertens@mpim-bonn.mpg.de}
\address{Department of Mathematics,
1420 Stevenson Center,
Vanderbilt University,
Nashville, TN 37240}
\email{larry.rolen@vanderbilt.edu}

\title[Periodic Taylor expansions of half-integral weight modular forms]{Periodicities for Taylor coefficients of half-integral weight modular forms}


\begin{document}
\begin{abstract}
Congruences of Fourier coefficients of modular forms have long been an object of central study. By comparison, the arithmetic of other expansions of modular forms, in particular Taylor expansions around points in the upper-half plane, has been much less studied. Recently, Romik made a conjecture about the periodicity of coefficients around $\tau_0=i$ of the classical Jacobi theta function $\theta_3$. Here, we  generalize the phenomenon observed by Romik to a broader class of modular forms of half-integral weight and, in particular, prove the conjecture.
\end{abstract}

\maketitle
\vspace*{-0.2in}

\section{Introduction}
Fourier coefficients of modular forms are well-known to encode many interesting quantities, such as  the number of points on elliptic curves over finite fields, partition numbers, divisor sums, and many more. Thanks to these connections, the arithmetic of modular form Fourier coefficients has long enjoyed a broad study, and remains a very active field today. However, Fourier expansions are just one sort of canonical expansion of modular forms. Petersson also defined \cite{Petersson} the so-called {\it hyperbolic} and {\it elliptic} expansions, which instead of being associated to a cusp of the modular curve, are associated to a pair of real quadratic numbers or a point in the upper half-plane, respectively. A beautiful exposition on these different expansions and some of their more recent connections can be found in \cite{ImamogluOSullivan}. In particular, there Imamoglu and O'Sullivan point out that Poincar\'e series with respect to hyperbolic expansions include the important examples of Katok \cite{Katok} and Zagier \cite{ZagierRealQuadratic}, which are the functions which Kohnen later used \cite{Kohnen} to construct the holomorphic kernel for the Shimura/Shintani lift. 

Here, we will focus on elliptic expansions, which are essentially  Taylor expansions. While a Fourier expansion of a given modular form $f\in M_k(\Gamma)$ for some weight $k$ and congruence subgroup $\Gamma\leq\SL_2(\Z)$ is an expansion at  a \emph{cusp} of $\Gamma$, i.e. at the boundary of the completed upper half-plane $\overline\HH:=\HH\cup\Q\cup\{\infty\}$, one might also consider expansions around an interior point $\tau_0\in\HH$. The classical Taylor expansion in the sense of complex analysis,
\[f(\tau)=\sum_{n=0}^\infty \left(\frac{d^nf}{d\tau^n}\right)(\tau_0)\cdot\frac{(\tau-\tau_0)^n}{n!},\]
only converges on an open disc of radius $y_0:=\Im(\tau_0)$ around $\tau_0$, which is not optimal because the natural domain of holomorphy of $f$ is the full upper half-plane $\HH$. Instead of this naive construction, one uses a Cayley-type transformation 
\[\tau\mapsto w=\frac{\tau-\tau_0}{\tau-\overline{\tau_0}}\]
to map the upper half-plane to the open unit disc, sending the point $\tau_0=x_0+iy_0$ to the origin, and consider $f$ as a function in $w$ instead. Taking the usual Taylor expansion with respect to $w$ around $w=0$ yields the relation
\begin{gather}\label{eqtaylor}
(1-w)^{-k}f\left(\frac{\tau_0-\overline{\tau_0}w}{1-w}\right)=\sum_{n=0}^\infty \partial^nf(\tau_0)\frac{(4\pi y_0w)^n}{n!}, \qquad (|w|<1),
\end{gather}
where 
\begin{gather}\label{eqraising}
\partial:=\partial_k:=D-\frac{k}{4\pi \Im(\tau)}\quad\text{with}\quad D:=\frac{1}{2\pi i}\frac{d}{d\tau}=q\frac{d}{dq},\quad (q:=e^{2\pi i\tau}),
\end{gather}
denotes the renormalized Maa{\ss} raising operator with the abbreviations $\partial^0=\operatorname{id}$ and $\partial^n:=\partial_k^n:=\partial_{k+2(n-1)}\circ...\circ\partial_{k+2}\circ\partial_k $ for $n>0$, see for instance \cite[Proposition 17]{Zagier}. Note that for any smooth function $f:\HH\to\C$ and $g\in\SL_2(\R)$ we have
\[(\partial_k f)|_{k+2}g=\partial_k(f|_k g)\]
where $|_k$ denotes the weight $k$ slash operator (see \Cref{secquasi} for the definition), so in particular the operator $\partial_k$ preserves modularity, but not holomorphy (except in weight $0$). 
\begin{remark*}
We note that for $k\notin\Z$, there is an ambiguity on the left-hand side of \eqref{eqtaylor}, while the right-hand side is well-defined for any $k$.  Since we have $(1-w)\neq 0$ for $|w|<1$ and the unit disc is simply connected, we can fix the branch of the holomorphic square-root that is positive for positive real arguments to make \eqref{eqtaylor} consistent for any half-integer $k\in\tfrac 12\Z$, as can be seen by restricting $w$ to the open interval $(-1,1)$ in the proof of \cite[Proposition 17]{Zagier}. 
\end{remark*}
It follows from the theory of complex multiplication that the coefficients in the Taylor expansion of a modular form with algebraic Fourier coefficients --- a condition we will assume throughout the paper if not specified otherwise ---  around a CM point (suitably normalized) are again algebraic numbers. In special cases, these are also known to have deep arithmetic meaning. For example, it was shown by Rodriguez-Villegas and Zagier that Taylor coefficients of Eisenstein series are essentially special values of Hecke $L$-functions \cite{VillegasZagier1}, a fact which later allowed them to give an explicitly computable  criterion to decide whether or not a prime $p\equiv 1\pmod 9$ is the sum of two rational cubes \cite{VillegasZagier2}, see also \cite[pp. 89-90 and pp. 97--99]{Zagier}.
\begin{remark*}
Loosely speaking, this relation between special values of $L$-functions and Taylor coefficients may already suggest their periodicity modulo primes in special cases since for instance in the simplest case of an $L$-function, the Riemann $\zeta$-function, the special values are essentially Bernoulli numbers, whose periodicity properties modulo primes are well-known.
\end{remark*}

Given these applications, it is natural to ask for arithmetic properties, for instance congruences, of Taylor expansions of modular forms. Works of the first author and Datskovsky \cite{DG} and of Larson and Smith \cite{LarsonSmith} have previously given conditions under which Taylor expansions of integral weight modular forms are periodic. Recently, Romik studied the Taylor coefficients of the
classical Jacobi theta function
\[\theta_3(\tau):=\sum_{n\in\Z} e^{\pi in^2\tau}\]
around the point $\tau_0=i$ \cite{Romik}. He gives explicit recursions for these coefficients and, based on numerical examples, he conjectures a certain behavior of these coefficients modulo primes. To be more precise, let $\Phi=\frac{\Gamma(1/4)^4}{8\pi^2\sqrt{2}}$ and define the numbers $d(n)$ by
\begin{gather}\label{eqRomik}
(1-w)^{-1/2}\theta_3\left(\frac{i+wi}{1-w}\right)=:\theta_3(i)\sum_{n=0}^\infty \frac{d(n)}{(2n)!}(\Phi w)^{2n},\qquad (|w|<1). 
\end{gather}
For instance, the first few values of $d(n)$ are $1,1,-1,51,849,-26199,\ldots$. Comparing this to \eqref{eqtaylor}, we point out that the derivatives $\partial^n\theta_3(i)$ vanish for odd $n$ since $i$ is a fixed point of the transformation $\tau\mapsto -1/\tau$, under which $\theta_3$ and all its non-holomorphic derivatives are equivariant. 
Romik  showed that the numbers $d(n)$ are all integers \cite[Theorem 1]{Romik} and posed the following conjecture.
\begin{conjecture}[{\cite[Conjecture 13]{Romik}}]\label{conjRomik}
Let $p$ be an odd prime. Then we have:
\begin{enumerate}
\item If $p\equiv3\pmod 4$, then $d(n)\equiv 0\pmod p$ for sufficiently large $n$.
\item If $p\equiv1\pmod 4$, the sequence $\left\{d(n)\pmod p\right\}_{n=1}^{\infty}$ is periodic.
\end{enumerate}
\end{conjecture}
In particular, regardless of the case, the sequence modulo $p$ is always eventually periodic. Romik also asks the question if a similar pattern persists modulo higher powers of primes \cite[Section 8]{Romik}.
Recently, part of \Cref{conjRomik} has been proven by Scherer \cite{Scherer}.
\begin{thm*}[{\cite[Theorem 1]{Scherer}}]\label{ThmScherer} \mbox{} 
\begin{enumerate}
\item  Part 1 of \Cref{conjRomik} is true. 
\item We have that $d(n)\equiv (-1)^{n+1}\pmod 5$.
\end{enumerate}
\end{thm*}
Apart from the congruences modulo $5$, part $2$ of Romik's conjecture remains open. In this paper, we prove and considerably generalize this half of the conjecture.

In order to state our main result we need to introduce an additional notation.
As usual, define the weight $k$ Eisenstein series for even integer $k>2$ by
\begin{gather}\label{eqEisenstein}
E_k(\tau):=1-\frac{2k}{B_k} \sum_{n\geq1} \sigma_{k-1}(n) q^n,
\end{gather}
where $B_k$ denotes the $k$th Benoulli number and $\sigma_{k-1}(n)=\sum_{d|n}d^{k-1}$.
Letting $\Theta(\tau) :=\sum_{n\in\Z}q^{n^2}$, the modular function $\phi_k:=E_{k}/\Theta^{2k}$ is modular on $\Gamma_0(4)$, and as such takes algebraic values at CM-points.
\begin{theorem}\label{thmmain}
Suppose that $k,N\in\N$ and let $f\in M_{k-1/2}(\Gamma_1(4N))$ be a modular form with algebraic integer Fourier coefficients. 
Further suppose that $p>3$ is a split prime in $\Q(\tau_0)$ for a CM point $\tau_0$. 

Assume furthermore that the absolute norm of the algebraic number $\phi_{p-1}(\tau_0)$ is $p$-integral and is not divisible by $p$.
Then there exists $\Omega \in \C^\times$, which can be chosen to depend only on $\tau_0$ and $p$, such that for $n_1,n_2>A$ satisfying
\[
n_1 \equiv n_2 \pmod{(p-1)p^A} 
\]
we have the congruence
\[
\partial^{n_1}f(\tau_0)/\Omega^{2k+4n_1-1} \equiv \partial^{n_2} f(\tau_0)/\Omega^{2k+4n_2-1} \pmod{p^{A+1}}.
\]
\end{theorem}
\begin{remark*}
The condition that $\phi_{p-1}(\tau_0)$ be a $p$-adic unit is entirely technical, and the theorem probably holds true without it. However, this condition simplifies our proof considerably, and so we have chosen to state the theorem in this way.
\end{remark*}

\begin{remark*}
The condition $n_1,n_2>A$ in the theorem originates from the application of the Euler-Fermat Totient Theorem in the proof. Therefore, our theorem predicts in complete generality when the sequence of Taylor coefficients of any half-integer weight modular form becomes periodic and what its maximal period length is. 
\end{remark*}




\begin{remark*}
It is worth noting that the inert prime case was studied in detail for integral weight forms by Larson and Smith \cite{LarsonSmith}. 
There, they found similar eventual vanishing results modulo $p$ as in part (1) of \Cref{conjRomik}. Although it appears numerically that 
more general versions of their work hold, it appears that new techniques are required to prove a general phenomenon since their proofs use
the structure of the algebra of integer weight modular forms on $\mathrm{SL}_2(\mathbb Z)$ in an essential way.
\end{remark*}

\begin{theorem}\label{corRomik}
Let $\tau_0\in\HH$ be a CM point such that the class number of  $K=\Q(\tau_0)$ is $1$. Assume further that the CM elliptic curve $E$ defined by $\C/\langle \omega,\omega\tau_0\rangle_\Z$ for a real period $\omega$ is defined over $\Q$ and the conditions and notations in \Cref{thmmain}. Then there exists a number $\widetilde\Omega\in\C^\times$ which only depends on $\tau_0$ such that for every prime $p>3$ that splits in $K$ and at which $E$ has good reduction we have the congruence
\[
\partial^{n_1}f(\tau_0)/\widetilde\Omega^{2k+4n_1-1} \equiv \partial^{n_2} f(\tau_0)/\widetilde\Omega^{2k+4n_2-1} \pmod{p^{A+1}}.
\]
for any $n_1,n_2>A$ with 
\[
n_1 \equiv n_2 \pmod{(p-1)p^A}.
\]
\end{theorem}
Part 2 of \Cref{conjRomik} follows by taking $f(\tau)=\Theta(\tau)\in M_{1/2}(\Gamma_0(4))$ as defined above and $\tau_0=i/2$ in \Cref{corRomik}. By combining with Scherer's result, this proves \Cref{conjRomik}.
\begin{cor}
Conjecture~\ref{conjRomik} is true. 
\end{cor}
\begin{remark*}
The assumptions on the class number and the elliptic curve $E$ in \Cref{corRomik} are again of a technical nature to simplify the proof and the statement of the result. 
\end{remark*}

The rest of this paper is organized as follows. In \Cref{secquasi} we collect some necessary background about quasimodular and almost holomorphic modular forms. \Cref{secKatz} contains the proof of \Cref{thmmain,corRomik}, which makes use of an important result following from the theory of Katz (see \Cref{DGlemma}). We conclude the paper by discussing examples of Taylor expansions of modular forms for $\Gamma_0(4)$ around various CM points in \Cref{secExamples}.

\section*{Acknowledgements}
The authors thank Robert Scherer and Don Zagier for useful discussions.
The second author's research has been supported by the European Research Council under European Union's Seventh Framework Programme (FP/2007-2013) / ERC Grant agreement n. 335220 - AQSER. 

\section{Quasimodular and almost holomorphic modular forms of half-integer weight}\label{secquasi}
In this section, we will review the basic theory of quasimodular and almost holomorphic forms, which we shall require in our proofs of the main results.
Quasimodular forms and almost holomorphic modular forms generalize classical modular forms. The first example of a quasimodular form is the Eisenstein series of weight $2$,
\[E_2(\tau):=1-24\sum_{n=1}^\infty n\frac{q^n}{1-q^n}=1-24\sum_{n=1}^\infty \sigma_1(n)q^n.\]
While $E_2$ is not modular, it very nearly is. In general, quasimodular forms have a slightly deformed modularity transformation, and  every quasimodular form has an associated almost holomorphic modular form. An almost holomorphic modular form is simply a modular form which, instead of being holomorphic, is a polynomial in $Y:=\frac{1}{-4\pi y}$, where $y:=\Im(\tau)$, with holomorphic functions as coefficients. In the case of $E_2$, the associated almost holomorphic modular form is the function
\[E_2^*(\tau):=E_2(\tau)+12Y,\]
which transforms as a modular form of weight $2$ on $\operatorname{SL}_2(\Z)$.
More precise definitions follow below.

The systematic study of these functions originates\footnote{Essentially the same concepts under slightly different names have been introduced independently by Shimura \cite{ShimuraQuasi}.} from work of Kaneko and Zagier \cite{KanekoZagier} on a theorem of Dijkgraaf \cite{Dijkgraaf}. In the last few years, these functions (in integral weight) have received a lot of attention in the context of the celebrated Bloch-Okounkov Theorem \cite{BlochOkou,ZagierBloch}.

In this section, we record special cases of Lemma 1.1 and Proposition 1.2 of \cite{Zemel}, where Zemel generalizes the concepts of quasimodular and almost holomorphic modular forms to the setting of real-analytic modular forms, possibly with singularities,  of arbitrary (real or complex) weights, arbitrary (vector-valued) multiplier systems for arbitrary Fuchsian groups.

We begin by recalling the slash operator. For a function $f\colon\HH\to\C$, a weight $k\in\tfrac 12\Z$, and a matrix $\gamma=\abcd\in\SLZ$, let
\[\left(f|_k\gamma\right)(\tau):=\begin{cases} (c\tau+d)^{-k} f\left(\frac{a\tau+b}{c\tau+d}\right) & \text{ if } k\in\Z, \\
                                              \left(\frac cd\right)\eps_d^{-2k} \left(\sqrt{c\tau+d}\right)^{-2k}f\left(\frac{a\tau+b}{c\tau+d}\right) & \text{ if } k\in\tfrac 12+\Z,
                                              \end{cases}\]
where for $k\in\tfrac 12+\Z$ we assume additionally that $\gamma\in\Gamma_0(4)$, i.e. $4\mid c$, $\left(\frac cd\right)$ denotes the extended Jacobi symbol in the sense of Shimura \cite{Shimura}, we choose the branch of the square root so that $-\pi/2<\arg\sqrt{z}\leq\pi/2$, which is consistent with the choice made in the remark following \eqref{eqtaylor}, and 
\[\eps_d=\begin{cases} 1 & \text{ if } d\equiv 1\pmod 4, \\ i & \text{ if } d\equiv 3\pmod 4.\end{cases}\]

With this notation in mind, we can make the following definition.
\begin{definition}
A \emph{quasimodular form} of weight $k\in\tfrac 12\Z$ and depth $\leq \sfd\in\N_0$ for $\Gamma\leq \SLZ$ ($\Gamma\leq \Gamma_0(4)$ if $k\notin\Z$) is a holomorphic function $f$ on $\HH$ with moderate growth when $\tau$ approaches any cusp in $\Q\cup\{\infty\}$ satisfying
\begin{gather}\label{eqquasitrans}
(f|_{k}\gamma)(\tau)=\sum_{j=0}^\sfd \left(\frac{c}{c\tau+d}\right)^j f_j(\tau)
\end{gather}
for all $\gamma=\left(\begin{smallmatrix} a & b \\ c & d \end{smallmatrix}\right)\in\Gamma$ and $\tau\in\HH$, where $f_0=f,...,f_\sfd$ are certain holomorphic functions, depending only on $f$ but not on 
$\gamma$, which satisfy the same growth conditions.
\end{definition}
We also say that the \emph{depth} of a quasimodular form $f$ is the largest integer $\sfd$ in \eqref{eqquasitrans}, such that $f_\sfd$ does not vanish identically. The space of quasimodular forms of weight $k$ and depth $\leq \sfd$ is denoted by $\widetilde{M}_k^{\leq \sfd}(\Gamma)$. If we allow arbitrarily large depth (which is actually at most $k/2$; see \Cref{propquasi}), we omit the superscript. 

A closely related notion is that of almost holomorphic modular forms of weight $k\in\tfrac 12\Z$, which are defined --- as mentioned at the beginning of this section --- as polynomials in $Y=\tfrac{1}{-4\pi y}$ with holomorphic coefficients, transforming like modular forms. The space of such functions is denoted by $\widehat{M}_{k}^{\leq \sfd} (\Gamma)$, where $\sfd$ denotes the maximal degree of the polynomial. Again, an omitted superscript indicates that the degree can be unbounded. The following proposition makes the aforementioned close connection between quasimodular forms and almost holomorphic modular forms explicit.
\begin{proposition}\label{propquasi}
Let $f\in \widetilde{M}_{k}^{\leq \sfd}(\Gamma)$ be a quasimodular form of weight $k$ and depth $\leq \sfd$ with corresponding functions $f_0,...,f_\sfd$ as in \eqref{eqquasitrans}. Then the following are true.
\begin{enumerate}
\item For $j=0,...,\sfd$ we have $f_j\in \widetilde{M}_{k-2j}^{\leq \sfd-j}(\Gamma)$, the corresponding functions being given by $\binom jr f_r$, $j\leq r\leq \sfd$. In particular, the function $f_\sfd$ is a modular form of weight $k-2\sfd$.
\item The function
\[f^*(\tau)=\sum_{j=0}^\sfd f_j(\tau)\left(\frac 1{2\pi i}Y\right)^{j}\]
transforms like a modular form of weight $k$. Conversely, if we have $G(\tau)=\sum_{j=0}^\sfd g_j(\tau)\left(\frac 1{2\pi i}Y\right)^{j}\in \widehat{M}_{k}^{\leq\sfd}(\Gamma)$, then $g_0\in\widetilde{M}_{k}^{\leq\sfd}(\Gamma)$ with corresponding functions $g_0,...,g_\sfd$.
\end{enumerate}
In particular, the graded rings $\widehat{M}(\Gamma)=\bigoplus_{k} \widetilde{M}_k(\Gamma)$ and $\widehat{M}_k(\Gamma)$ are canonically isomorphic.
\end{proposition}
In the case of integer weight, this proposition goes back to \cite{KanekoZagier}, for half-integer weight it is, as mentioned earlier, a special case of Lemma 1.1 and Proposition 1.2 of \cite{Zemel}. 
We record the following version of \cite[Proposition 20]{Zagier}. The proof of this result carries over almost literally, making occasional use of \Cref{propquasi}; thus, we omit it here. 
\begin{proposition}\label{Prop20V}
The following are true.
\begin{enumerate}
\item The differential operator $D$ maps quasimodular forms to quasimodular forms, i.e., for $f\in\tM_k^{\leq\sfd}(\Gamma)$, we have $Df\in \tM_{k+2}^{\leq\sfd+1}(\Gamma)$.
\item Every quasimodular form is a polynomial in $E_2$ whose coefficients are modular forms, i.e., we have a decomposition $\tM_k^{\leq\sfd}(\Gamma)=\bigoplus_{j=0}^\sfd M_{k-2j}(\Gamma)\cdot E_2^j$.
\item Every quasimodular form is a linear combination of derivatives of modular forms and derivatives of $E_2$. More precisely, we have
\[\tM_k^{\leq\sfd}(\Gamma)=\begin{cases} \bigoplus_{j=0}^\sfd D^j(M_{k-2j}(\Gamma)) & \text{ if } \sfd<k/2, \\
\bigoplus_{j=0}^{k/2-1} D^j(M_{k-2j}(\Gamma))\oplus \C\cdot D^{k/2-1}E_2\ & \text{ if } \sfd=k/2. \end{cases}\]
\end{enumerate}
\end{proposition}
In the proof of \Cref{thmmain}, we require the following easy consequence of the above.
\begin{corollary}\label{corquasi}
Let $H\in M_k(\Gamma)$ and $G\in M_\ell(\Gamma)$, $k,\ell\in\tfrac 12\Z$. Then we have that $G\cdot(D^n H)\in\widetilde{M}_{k+\ell+2n}(\Gamma)$ and the associated almost holomorphic modular form is given by $G\cdot (\partial^n H)$. 
\end{corollary}
\begin{proof}
It is clear that it suffices to show that the almost holomorphic modular form associated to $D^n H$ is given by $\partial^n H$. As remarked above, we may apply \Cref{Prop20V} in this setting, wherefore $D^n H\in \widetilde{M}_{k+2n}(\Gamma)$. Furtherfore, $\partial^n H$ is an almost holomorphic modular form of the same weight whose constant term with respect to $Y$ is precisely $D^n H$, as one sees immediately  from the following formula for the iterated raising operator, which is easily shown by induction (see for instance \cite[Equation (56)]{Zagier}), 
\[\partial_k^n =\sum_{m=0}^n(-1)^{n-m}\binom nm \frac{(k+n-1)!}{(k+m-1)!} Y^{n-m} D^{m}.\]
\end{proof}


\section{Proofs of \Cref{thmmain} and \Cref{corRomik}]}\label{secKatz}
In this section, we will prove the main results. 
\subsection{Preliminary results and work of Katz}
The periodicity phenomenon in \Cref{thmmain} is ultimately a consequence of the very general theory of Katz \cite{Katz}. However, Katz's work does not contain a statement which is exactly sufficient for our purposes here. 
The key result we need is \Cref{DGlemma} below which is an extension  of Lemma 1 from \cite{DG}.
This statement, as well as the theory developed in \cite{Katz}, is  formulated for the case of integral weight modular forms, and all weights are assumed to be integral throughout this subsection. Also, $p$ is always assumed to be a prime larger than $3$.


The mantra we need here, which 
requires some work for its precise specialization which we will employ later,
is that that {\it $p$-adically close modular forms have $p$-adically close values}.
That is nothing but a specialization of the {\it $q$-expansion principle} from \cite[Section 5.2]{Katz}. 

To make this precise, we first of all need a version of Damerell's theorem which allows for making all quantities under consideration algebraic.
The idea is simple: while {\it quasimodular forms} have $q$-expansions, {\it almost holomorphic modular forms} take essentially algebraic values at $\tau_0$ (see \Cref{Damerell} below), and as described in \Cref{propquasi}, the two rings
are canonically isomorphic. Thus, we can assign algebraic values to algebraic $q$-expansions in order to study congruences between them.
 \Cref{Prop20V} allows us to assign a $q$-expansion to every  quasimodular form $f\in \widetilde{M_*}(\Gamma)$: we simply plug in
the $q$-expansions of modular forms and $E_2$ into the expression. Namely, for $g\in \widetilde{M}_k(\Gamma)$, \Cref{Prop20V} (2) implies
\[
g=\sum_{r=0}^{\lfloor k/2 \rfloor} F_{k-2r} E_2^r \in \C \llbracket q \rrbracket \hspace{3mm} \text{with} \hspace{3mm} F_{k-2r} \in M_{k-2r}(\Gamma).
\]
We will identify $g\in \widetilde{M}_k(\Gamma)$ with the associated {\it almost holomorphic} form $g^* \in \widehat{M}_k(\Gamma)$ via the isomorphism between $\widetilde{M}_*(\Gamma)$ and
$\widehat{M}_*(\Gamma)$ which preserves the gradation, and set the value
\[
g^*(\tau_0)=\sum_{r=0}^{\lfloor k/2 \rfloor} F_{k-2r}(\tau_0) (E^*_2(\tau_0))^r \in \C \hspace{3mm} \text{with} \hspace{3mm} F_{k-2r} \in M_{k-2r}(\Gamma).
\]

From now on, we fix an algebraic number field $K$ which is large enough to contain the relevant quantities below, and we denote its ring of integers by ${\cO}$.

With these notations, we have the following algebraicity statement.

\begin{proposition}[Katz's version of Damerell's theorem, {\cite[Theorem 4.0.4]{Katz}}] \label{Damerell} 
\mbox{} 

If $\tau_0 \in K$,
then there exists an $\omega \in \C^*$ such that 
\[
\text{if $g\in \widetilde{M}_k(\Gamma) \cap K\llbracket q \rrbracket$, then $g^*(\tau_0)/\omega^k \in K$}.
\]
\end{proposition}

We now pass to the question about congruences. 
Given two formal power series $g_1=\sum_{n=0}^\infty b_1(n) q^n,\ g_2=\sum_{n=0}^\infty
b_2(n) q^n\in K\llbracket q\rrbracket$, we say that $g_1 \equiv g_2 \pmod {p^A}$ if their coefficients are congruent modulo $p^A$, i.e. if $b_1(n) - b_2(n) \in p^A\cO$ for all $n$. 
This notation applies, in particular, to the situation when $g_1$ and, $g_2$ are (the $q$-expansions of) quasimodular forms in  $ \widetilde{M_*}(\Gamma) \cap K\llbracket q \rrbracket$.

Clearly, if $\omega \in \C^\times$ works in \Cref{Damerell} then so does any $K^\times$-multiple.
Note furthermore that if $\omega \in \C^\times$ satisfies \Cref{Damerell} for any single $g\in \widetilde{M}_k(\Gamma) \cap K\llbracket q \rrbracket$ then so it also does 
for all $g\in \widetilde{M}_k(\Gamma) \cap K\llbracket q \rrbracket$, and we will make a specific choice now.

The discussion above puts no restrictions on the prime $p$ under consideration. From now on, we assume that $p$ splits in $\Q(\tau_0)$. 
The first consequence of this choice is that $E_{p-1}(\tau_0) \neq 0$ (see Section 2.1 of \cite{KaM}), which allows us to pick $\omega \neq 0$ in the following proposition.

This proposition is nothing but a specialization to our notations of the  (more general) $q$-expansion principle from \cite[Section 5.2]{Katz} combined with a $p$-adic version of 
Damerell's theorem \cite[Comparison Theorem 8.0.9]{Katz} It was formulated and proved as \cite[Lemma 1]{DG}  in the special case when $N=1$. 
For the general case which we need here, the proof follows mutatis mutandis; we have omitted this simple translation of the proof given in \cite{DG} for notational simplicity.

\begin{prop}[{\cite[Lemma 1]{DG}}] \label{DGlemma}
Assume that $p$ splits in $\Q(\tau_0)$.
Pick a complex number $\omega_p$ so that $\omega_p^{p-1} = E_{p-1}(\tau_0)$. 
For $i=1,2$ let 

\[
g_i=\sum_{n=0}^\infty b_i(n) q^n \in  \widetilde{M}_{k_i} (\Gamma)\cap \cO\llbracket q \rrbracket.
\]
If
\[
g_1 \equiv g_2 \pmod{ p^A}
\]
for a positive integer $A$, then
\[
g_1^*(\tau_0)/\omega_p^{k_1} \equiv g_2^*(\tau_0)/\omega_p^{k_2} \pmod{ p^A}.
\]
\end{prop}
\begin{rem} A naive explanation for the choice of $\omega_p$ is as follows. Since $E_{p-1} \equiv 1 \pmod p$ by the von Staudt-Clausen Theorem, we also ought to have $E_{p-1}(\tau_0) \equiv 1 \pmod p$.
In other words, if \Cref{DGlemma} is true for some choice of $\omega_p$, then this should be a correct choice. Note however that the proposition is simply false as stated for
inert primes although it may still happen that $E_{p-1}(\tau_0) \neq 0$: For example, the prime $13$ is inert in $\Q(\sqrt{7})$, but for $\tau_0=\frac{1+i\sqrt{7}}{2}$ we have $E_{12}(\tau_0)\approx 0.98818418\neq 0$. Now the two weight $12$ modular forms $E_{12}$ and $E_{12}+13\Delta$ with $\Delta:=(E_4^3-E_6^2)/1728$ denoting the usual $\Delta$ function (which has integer Fourier coefficients)  are obviously congruent modulo $13$, but choosing $\omega_p$ as specified in \Cref{DGlemma}, we find that
\[E_{12}(\tau_0)/\omega_p^{12}=1\qquad\text{and}\qquad (E_{12}(\tau_0)+13\Delta(\tau_0))/\omega_p^{12}=\frac{211934}{212625}\equiv 6\pmod{13}.\]
\end{rem}

\subsection{Multiplication by the $\Theta$-function and passage to half-integral weight}
Here we sketch how to generalize the results of the preceding subsection to half-integral weight.
A generalization of Katz's theory to half-integral weight
 has also been developed by Ramsey \cite{Ramsey}. Although it is based on similar ideas, Ramsey's generalization is less explicit than our approach here,
and is not intended in the specific direction which we require, and so is less convenient for our purposes.
To move from integral weight to half-integral weight, we use the simple (and common) technique of multiplying by Jacobi's $\Theta(\tau)$. Thanks to 
Jacobi's identity
\begin{equation}
 \label{etas}
\Theta(\tau) = \frac{\eta^5(2\tau)}{\eta^2(\tau) \eta^2(4\tau)},
\end{equation}
where $\eta(\tau):=q^{1/24}\prod_{n=1}^\infty (1-q^n)$ denotes the Dedekind eta function, $\Theta$ does not vanish in the interior of the upper half-plane.
We then need the following result on the action of this multiplication by $\Theta$ operation on quasimodular forms.
\begin{lemma} \label{diff}
Let $H \in \C\llbracket q \rrbracket$ be such that the product 
\[
\Theta H \in \widetilde{M}_k(\Gamma)
\] 
is (a $q$-expansion of) a quasimodular form 
of weight $k\in\Z$ on $\Gamma=\Gamma_1(N)$, where $4|N$. 
Then 
\[
\Theta DH \in \widetilde{M}_{k+2}(\Gamma).
\]
\end{lemma}

\begin{proof}
Since $\Theta H \in \widetilde{M}_k(\Gamma)$, we have that
\[
D(\Theta H) = \Theta DH + H D\Theta \in \widetilde{M}_{k+2}(\Gamma),
\]
and it suffices to show that $H D\Theta \in \widetilde{M}_{k+2}(\Gamma)$.
It follows from \eqref{etas} that
\[
\begin{split}
24\frac{D\Theta}{\Theta}(\tau) & = 10E_2(2\tau)-2E_2(z)-8E_2(4\tau) \\ & = 10\left(E_2(2\tau)-\frac{1}{2} E_2(\tau)\right)
-8\left(E_2(4\tau)-\frac{1}{4}E_2(\tau)\right) +E_2(\tau) \in \widetilde{M}_{k+2}(\Gamma_0(4)),
\end{split}
\]
and
therefore
\[
H D\Theta = H \frac{D\Theta}{\Theta} \Theta \in \widetilde{M}_{k+2}(\Gamma)
\]
as required.
\end{proof}

\subsection{Periodicity of Taylor coefficients }
We now have all the pieces in place to prove our main results.
\begin{proof}[Proof of \Cref{thmmain}]
Let $p$ be a splitting prime in  $\Omega_p\in\C^\times$ be such that $\Omega_p^2=\omega_p$ with $\omega_p$ as in \Cref{DGlemma}. Suppose further that 
\[f\in M_{k-1/2}(\Gamma) \cap \cO\llbracket q \rrbracket\] 
is a half-integral weight modular form with algebraic integer Fourier coefficients, and assume that both
$f(\tau_0)/\Omega_p^{2k-1}$ and $\Theta(\tau_0)/\Omega_p$ lie in $K$.
By Euler's Totient Theorem, if both $n_1,n_2>A$, we have 
\[
n_1 \equiv n_2 \pmod {(p-1)p^A} \hspace{3mm} \text{implies} \hspace{3mm} D^{n_1}(f) \equiv D^{n_2}(f) \pmod {p^{A+1}}.
\]

Multiplication by $\Theta$ will preserve these congruences:
\[
 \Theta D^{n_1}f \equiv \Theta D^{n_2}f \pmod{ p^{A+1}}.
\]
\Cref{diff} (applied repeatedly) implies that both products 
\[
\Theta  D^{n_1}f  \in \widetilde{M}_{k+2n_1}(\Gamma) \hspace{3mm} \textup{and} \hspace{3mm} \Theta  D^{n_1}f \in   \widetilde{M}_{k+2n_2}(\Gamma),
\]
are quasimodular forms and we can apply \Cref{DGlemma} to derive the congruence 
\[
\left( \Theta  D^{n_1}f \right)^* (\tau_0) /\omega_p^{k+2n_1} \equiv \left( \Theta  D^{n_2}f \right)^* (\tau_0) /\omega_p^{k+2n_2}  \pmod {p^{A+1}}
\]
for the (normalized) values at $\tau_0$.
We now apply \Cref{corquasi} to evaluate the quasimodular forms
$ \Theta  D^{n_1}(f)$ and  $\Theta  D^{n_2}(f)$ of integral weight at $\tau_0$
\[
\left( \Theta  D^{n_i}f \right)^* (\tau_0)  = \Theta(\tau_0)  \partial^{n_i}f (\tau_0) \hspace{3mm} \textup{for} \hspace{3mm} i=1,2,
\]
and factor out $\Theta(\tau_0)$, which by \eqref{etas} is not $0$.
The extra assumption on $p$-integrality of the value $\phi_{p-1}(\tau_0)$   allows us to guarantee that
\[
v_p \left(\Nm^K_{\Q} \left(\Theta^2(\tau_0)/\omega \right) \right) = 0,
\]
where $\Nm$ is the norm map, and $v_p$ is the $p$-adic valuation. We then can cancel this quantity, and obtain the desired periodicity modulo powers 
of the splitting prime:
\begin{equation*} \label{final_congr}
\partial^{n_1}f (\tau_0)/\Omega^{2k+4n_1-1} \equiv  \partial^{n_2}f (\tau_0)/\Omega^{2k+4n_2-1} \pmod{ p^{A+1}}.
\end{equation*}
\end{proof}


\subsection{Deligne's congruence and proof of \Cref{corRomik}}

So far,  \Cref{thmmain} claims the existence of $\Omega_p \in \C^\times$ which depends on the splitting prime $p$.
However, the conjecture of Romik is stated for a global choice, common for all primes. In this subsection, we show how to make a global choice,
and compare that with the choice made by Romik in \cite{Romik}. The fact that these two choices differ by a $p$-adic unit for every splitting prime $p>2$ will
allow us to derive  \Cref{conjRomik} from our  \Cref{thmmain}.

\begin{proof}[Proof of \Cref{corRomik}]
Let $K=\Q(\tau_0)$ be an imaginary quadratic field.
Define $\omega=\omega_{\tau_0}$ to be the real period of the  CM elliptic curve $E=\C/\langle\omega,\omega \tau_0\rangle_\Z$,
and let 
\[
\wp(z) := \frac{1}{z^2} + \sum_{n\geq 2} c_n z^{2n-2}
\]
be the associated Weierstrass $\wp$-function.

By assumption, the elliptic curve $E$ is defined over $\Q$ and the class number of $K$ is $1$, which implies that $c_n \in \mathbb Q$, where 
\[
c_n=(2n-1) \omega^{-2n} \sum_{(0,0) \neq (m,n)\in \Z \times \Z}\frac{1}{(m\tau_0+n)^{2n}}.
\]
These quantities are nothing but the values of Eisenstein series at $\tau_0$, properly normalized.
Namely (cf. \cite[Section 2.2]{Zagier} for the notations and the normalizations of Eisenstein series $\mathbb G_k$ and $E_k:=-\frac{2k}{B_k} {\mathbb G}_k$), we have
\[
c_n= 2 \left( \frac{2 \pi i}{\omega} \right)^{2n} \frac{1}{(2n-2)!} {\mathbb G}_{2n}(\tau_0).
\]

We now define Bernoulli-Hurwitz numbers $BH(2n)$ for integers $n\geq 1$ following \cite{KatzBH} as 
\[
c_n=:\frac{BH(2n)}{2n} \frac{1}{(2n-2)!},
\]
and with the above notations we have that
\[
BH(2n) = - \left( \frac{2\pi i}{\omega} \right)^{2n} B_{2n} E_{2n}(\tau_0),
\]

By assumption, $E$ has good reduction at $p$ and 
denote now by $A(p) \in \mathbb F_p = \mathbb Z/p\mathbb Z$ the Hasse invariant  of its modulo $p$ reduction.
All we need to know here is that $A(p) \neq 0$ if and only if the elliptic curve has good ordinary reduction at $p$, 
i.e., the prime $p$ splits in $\Q(\tau_0)$.

We now quote a special case of the part 1 of the theorem proved in \cite{KatzBH} :
\[
p \cdot BH(p-1) \equiv A(p) \pmod p.
\]
We translate this congruence using the above notations into
\[
-pB_{2n} \left( \frac{2\pi i}{\omega} \right)^{2n} E_{p-1}(\tau_0) \equiv A(p) \pmod p
\]
which simplifies using the von Staudt-Clausen congruence $pB_{p-1} \equiv -1 \pmod p$ into
\begin{equation} \label{deligne}
\left( \frac{2\pi i}{\omega} \right)^{p-1} E_{p-1}(\tau_0) \equiv A(p) \pmod p,
\end{equation}
that is the left-hand side is an algebraic integer which is non-zero modulo $p$ if and only if $p$ splits in $\Q(\tau_0)$.

We now compare the local choice of $\omega_p$ from \Cref{DGlemma} which was $\omega_p^{p-1} = E_{p-1}(\tau_0)$ with the global (i.e. independent of $p$) 
$\omega$ in  \eqref{deligne}, and conclude that the ratio of two omegas is a $p$-adic unit as we wanted. This completes the proof of \Cref{corRomik}.
\end{proof}
\begin{remark*}
In the case when $\tau_0=i$, congruence \eqref{deligne} was proved by Hurwitz in \cite{Hurwitz}. 
The above exposition follows closely Katz's paper \cite{KatzBH}, where a short proof based on the $q$-expansion principle of more general congruences  is presented.
An independent and elementary proof is presented by Kaneko and Zagier in \cite[Section 3]{KaZa} where a slightly different normalization for Eisenstein series is chosen.
In this paper (and in many others, in fact), the congruence is attributed to Deligne. 
\end{remark*}
\begin{remark*}
In the case when $\tau_0=i/2$, which is the objective of part (2) of \Cref{conjRomik}, it is classical \cite{Hurwitz}\footnote{In loc. cit., the computation is done for the CM point $\tau_0=i$, but one can use the same method to get the result for $i/2$.}  (or see \cite[Section 9.6]{Husemoller}) that
\[
\omega= \frac{1}{2} \sqrt{\pi} \frac{\Gamma(1/4)}{\Gamma(3/4)}=2\int_0^1 \frac{dx}{\sqrt{1-x^4}} = 2.62205755429211 \ldots
\]
Note that compared to Romik's choice of normalization, we find that $\omega^4=2\pi^2\Phi^2$. The factor of $2$ is of no importance as it is a $p$-adic unit for any odd prime and the additional power of $\pi$ originates from the different normalizations of the series defined in \eqref{eqtaylor} and  \eqref{eqRomik}.

The elliptic curve $\C/\langle\omega,\omega \tau_0\rangle_\Z$ in this special case has Weierstrass equation
\[
y^2=4x^3 -44x+56 \hspace{3mm} \text{with the non-vanishing differential $dx/y$.}
\]
\end{remark*}


\section{Examples}\label{secExamples}
In this section, we present several examples for the periodicity of Taylor coefficients at two different CM points, $\tau_0=i$
and $\tau_0=\z_7=\frac{1+i\sqrt{7}}{2}$. For the sake of being completely explicit, we focus on modular forms for the group $\Gamma_0(4)$. It is a well-known fact, which is easily verified using the dimension formula for spaces of modular forms for this group, that the algebra of modular forms for this group is a free polynomial algebra on two generators. More precisely, we have 
\[M_*(\Gamma_0(4))=\C[\Theta,F_2],\]
where the usual Jacobi theta function $\Theta(\tau)$ was defined in \eqref{etas} and 
\[F_2(\tau):=\eta(4\tau)^8\eta(2\tau)^{-4}=\sum_{n\,odd} \sigma_1(n)q^n\] 
is a weight $2$ modular form  (see for instance \cite{Cohen}). Note that in odd integer weight $k$, we include the spaces $M_k(\Gamma_0(4),\chi_{-4})$ transforming with the non-trivial Nebentypus modulo $4$ rather than those with trivial Nebentypus, which would be empty anyway.

It follows from the general theory of complex multiplication (cf. \cite[Proposition 26 and p. 84]{Zagier}) that the values of any weight $k$ modular form for $\Gamma_0(4)$ at a CM point $\tau_0$ of (fundamental) discriminant $D$ are algebraic multiples of $\Omega_D^k$, where 
\[\Omega_D=\frac{1}{\sqrt{2\pi |D|}}\left(\prod\limits_{j=1}^{|D|-1}\Gamma(j/|D|)^{\chi_D(j)}\right)^\frac{1}{2h'(D)},\]
$\chi_D=\left(\frac{D}{\cdot}\right)$ denotes the Kronecker character and $h'(D)$ denotes the modified class number of discriminant $D$, i.e. the number of $\SLZ$-equivalence classes of integral positive definite binary quadratic forms of discriminant $D$ multiplied by $1/3$ or $1/2$ if $D=-3$ or $D=-4$ resp. Indeed we find that
\begin{gather}
\label{eqvaluesi}
\Theta(i)=\sqrt[4]{\frac{3+2\sqrt{2}}{2}}\Omega_{-4}^{1/2} \quad\text{and}\quad F_2(i)=\frac{3-2\sqrt{2}}{32}\Omega_{-4}^2, 
\end{gather}
\begin{gather}
\label{eqvaluesz7}
\Theta(\z_7)=\sqrt[4]{\frac{8+3\sqrt{7}}{4}}\Omega_{-7}^{1/2} \quad\text{and}\quad F_2(\z_7)=-\frac{8-3\sqrt{7}}{2^{6}}\Omega_{-7}^2.
\end{gather}

Closely following the proof of \cite[Proposition 28]{Zagier} we offer the next two propositions which allow us to compute the Taylor coefficients of any modular form for $\Gamma_0(4)$ at one of the points $i$ and $\z_7$ recursively. This method can be used completely analogously for Taylor coefficients at any other CM point, which is also why we only give a detailed proof of \Cref{propTaylori}. Generalizing the method to other groups than $\Gamma_0(4)$ is also possible, but some care must be taken if the algebra of modular forms in question is not a free polynomial algebra, which it usually is not.

Before formulating the propositions, we introduce the following modification of the Serre derivative (see \cite[Equation (67)]{Zagier}\footnote{Note that in loc. cit., there is a slight typographical error in that the additional application of $\vartheta_\phi$ to $\vartheta_\phi^{[n]}f$ in the definition of $\vartheta_\phi^{[n+1]}f$ is omitted there.}). Let $\phi$ be any quasimodular form of weight $2$ for $\Gamma_0(4)$ such that the associated almost holomorphic modular form is given by $\phi^*(\tau)=\phi(\tau)-\frac{1}{4\pi y}$, hence transforms like a modular form of weight $2$. The Eisenstein series $\tfrac{1}{12}E_2$ for instance would be a valid choice, but not always the most convenient one, as illustrated for instance in \Cref{propTaylorz7}. Then we define the modified Serre derivative by
\begin{gather}
\vartheta_\phi f:=Df-k\phi f
\end{gather}
for $f\in M_k(\Gamma_0(4))$. This function maps $M_k(\Gamma_0(4))$ to $M_{k+2}(\Gamma_0(4))$, like the usual Serre derivative. The iterated version of this operator $\vartheta_\phi^{[n]}\colon M_k(\Gamma_0(4))\to M_{k+2n}(\Gamma_0(4))$ is then defined recursively via
\begin{gather}\label{eqmodifiedSerre}
\vartheta_\phi^{[0]}f=f,\quad \vartheta_\phi^{[1]}f=\vartheta_\phi f,\quad\vartheta_\phi^{[n+1]}=\vartheta_\phi(\vartheta_\phi^{[n]} f)+n(k+n-1)\psi \vartheta_\phi^{[n-1]} f\quad (n\geq 1),
\end{gather}
where $\psi\in M_4(\Gamma_0(4))$ is given by $\psi=D\phi-\phi^2$. In the special case for instance where $\phi=\tfrac{1}{12}E_2$, we have $\psi=-\tfrac{1}{144}E_4$. In this particular case, we omit the subscript of the operator, so $\vartheta^{[n]}:=\vartheta^{[n]}_{\frac 1{12}E_2}$.

Our first proposition now gives the claimed recursion for the Taylor coefficients of a modular form at the point $i$.
\begin{proposition}\label{propTaylori}
Let $f\in M_k(\Gamma_0(4))$ with $k\in\frac12\Z$ and let $P(X,Y)\in\C[X,Y]$ be a polynomial such that $P(\Theta,F_2)=f$. Then 
\[\partial^nf(i)=\left(\frac{3+2\sqrt{2}}{2}\right)^{n+k/2}p_n\left(\frac{17-12\sqrt{2}}{16}\right)\Omega_{-4}^{2n+k}=p_n\left(\frac{17-12\sqrt{2}}{16}\right)\Theta(i)^{4n+2k},\]
where $p_n(t)$ is the polynomial defined recursively by
\begin{align*}
p_{-1}(t)&=0,\quad p_0(t)=\frac{P(X,tX^4)}{X^{2k}},\\
p_{n+1}(t)&=\frac{1}{24}(80t-1)(2k+4n)p_n(t)-(16t^2-t)p_n'(t)\\
          &\qquad\qquad\qquad\qquad-\frac{1}{144}n(n+k-1)(256t^2+224t+1)p_{n-1}(t)\ (n\geq 0).
\end{align*}
\end{proposition}
\begin{proof}
Since the completed weight $2$ Eisenstein series $E_2^*$ vanishes at $i$, it follows by comparing the associated Cohen-Kusnetsov series (see \cite[Equation (68)]{Zagier}) that $\partial^nf(i)=\vartheta^{[n]}f(i)$ for all $n$. Since $\vartheta^{[n]}$ maps modular forms of weight $k$ to modular forms of weight $k+2n$, we can view $\vartheta^{[n]}$ as an operator on the polynomial ring $\C[\Theta,F_2]$. In particular, there is a polynomial $P_n(X,Y)\in\C[X,Y]$ such that $\vartheta^{[n]}f=P_n(\Theta,F_2)$. Explicitly, we compute
\begin{align*}
\vartheta \Theta &=\frac{1}{24}(80F_2\Theta-\Theta^5),\\
\vartheta F_2&= \frac{1}{6}(5\Theta^4F_2-16F_2^2),\\
E_4&= \Theta^8+224\Theta^4F_2+256 F_2^2.
\end{align*}
Hence we can write
\[\vartheta = \frac{1}{24}(80F_2\Theta-\Theta^5)\frac{\partial}{\partial \Theta}+\frac{1}{6}(5\Theta^4F-16F_2^2)\frac{\partial}{\partial F_2},\]
which yields the following recursion for $P_n$: 
\begin{align*}
P_{-1}(X,Y)&=0,\quad P_0(X,Y)=P(X,Y),\\
P_{n+1}(X,Y)&=\frac{1}{24}(-X^5+80XY)\frac{\partial P_n(X,Y)}{\partial X}+\frac{1}{6}(5X^4Y-16Y^2)\frac{\partial P_n(X,Y)}{\partial Y}\\
&\qquad\qquad\qquad -\frac{1}{144}n(n+k-1) (X^8+224X^4Y+256Y^2) P_{n-1}(X,Y).
\end{align*}
The polynomials $P_n(X,Y)$ are weighted homogeneous of weight $k+2n$, where $X$ has weight $1/2$ and $Y$ has weight $2$. Thus we can write $P_n(X,Y)=X^{4n+2k} p_n(Y/X^4)$, where $p_n(t)\in\C[t]$ is a single variable polynomial. Since
\begin{align*}
\frac{\partial P_n(X,Y)}{\partial X}
                                    &=X^{4n+2k-1}[(4n+2k)p_n(Y/X^4)-4(Y/X^4)p_n'(Y/X^4)]
\end{align*}
and 
\[\frac{\partial P_n(X,Y)}{\partial Y}=X^{4n+2k-4}p_n'(Y/X^4),\]
we find the following differential recursion for $p_n$:
\begin{align*}
p_{n+1}(t)
          &=\frac{1}{24}(80t-1)(4n+2k)p_n(t)-(16t^2-t)p_n'(t)\\
          &\qquad\qquad\qquad\qquad -\frac{1}{144}n(n+k-1) (256t^2+224t+1) p_{n-1}(t),
\end{align*}
i.e. together with the two  stated initial values the recursion claimed. Thus we have shown that
\[\vartheta^{[n]} f = \Theta^{4n+2k}p_n(F_2/\Theta^4),\]
hence
\[\partial^n f(i)=\vartheta^{[n]} f(i)=\Theta(i)^{4n+2k}p_n(F_2(i)/\Theta^4(i)),\]
which yields the claim using the values given in \eqref{eqvaluesi}.
\end{proof}
As an appplication of \Cref{propTaylori}, we offer the following example.
\begin{example}\label{exTaylori}
The Taylor coefficients of $\Theta$ at $\tau_0=i$ are given as follows,
\[(1-w)^{-1/2}\Theta\left(i\frac{1+w}{1-w}\right)=\Theta(i)\sum_{n=0}^\infty \frac{c(n)}{n!}(\Phi w)^n,\qquad (|w|<1),\]
where we choose $\Phi=\eps^4\pi\Omega_{-4}^2=\frac{(17+12\sqrt{2})\Gamma(1/4)^4}{16\pi^2}$, and where $\eps=1+\sqrt{2}$ is the fundamental unit in $\Q(\sqrt 2)$. Concretely, we compute the following table of values from the recursion in \Cref{propTaylori}.
\begin{center}
\begin{table}[h!]
\begin{tabular}{|c||c|c|c|c|c|c|c|c|c|c|c|c|c|}
\hline
$n$    & $0$ & $1$    & $2$ & $3$      & $4$  & $5$     & $6$        & $7$        & $8$     & $9$         & $10$      & $11$          \\
\hline
$c(n)$ & $1$ & $\eps$ & $1$ & $-3\eps$ & $17$ & $9\eps$ & $-111$ & $2373\eps$ & $12513$ & $86481\eps$ & $-146079$ & $-9806643\eps$ \\
\hline
\end{tabular}
\end{table}
\end{center}
The number $\Phi$ here has been chosen in order to make the coefficients $c(n)$ integers in $\Q(\sqrt 2)$, which one may verify by a straightforward induction argument.
In view of \Cref{DGlemma}, the period should be chosen depending on the prime modulus $p$ in order to find periodicity, but for the sake of uniformity, we keep this choice of period. By Fermat's Little Theorem, this still results in a periodic sequence modulo $p$ but with a longer period than with the choice in \Cref{DGlemma}. This motivates the notation $\overline{a_1,...,a_\ell}^b$ as a shorthand for 
\[a_1,...,a_\ell,ba_1,...,ba_\ell,b^2a_1,...,b^2a_\ell,...,\]
i.e. the \emph{quasiperiod} $a_1,...,a_\ell$ is multiplied by $b$ in each repetition. In other words, multiplying the chosen transcendental factor $\Phi$ by an $\ell$th root of $b$ yields an actually periodic coefficient sequence. 

Considering the first $200$ coefficients we find that
\begin{align*}
\{c(n)\}_{n=0}^\infty &\equiv \{1,\ \overline{\eps,\ 1}^2\} \pmod{5},\\
                      &\equiv \{1,\ \overline{\eps,\ 1,\ -3\eps,\ -8,\ 9\eps,\ -11,\ -2\eps, -12,\ 6\eps,\ -4}^7\}\pmod{5^2},
\end{align*}
and that $c(n)\equiv 57c(n+50)\pmod{5^3}$ for $n\geq 11$.

For $p=13$, we obtain
\begin{align*}
\{c(n)\}_{n=0}^\infty & \equiv \{1,\ \overline{\eps,1,-3\eps,-8,9\eps,-11,-2\eps,-12,6\eps,-4}^7\}\pmod{13}.
\end{align*}
\end{example}
With only a small alteration, we obtain the analogous result for the point $\z_7$.
\begin{proposition}\label{propTaylorz7}
Let $f\in M_k(\Gamma_0(4))$ with $k\in\frac12\Z$ and $P\in\C[X,Y]$ such that $f=P(\Theta,F_2)$. Then 
\begin{align*}
\partial^nf(\z_7)&=\left(\frac{8+3\sqrt{7}}{4}\right)^{n+k/2}q_n\left(-\frac{127-48\sqrt{7}}{16}\right)\Omega_{-7}^{2n+k}\\
&=q_n\left(-\frac{127-48\sqrt{7}}{16}\right)\Theta(\z_7)^{4n+2k}
\end{align*}
where $q_n(t)$ is defined recursively by 
\begin{align*}
q_{-1}(t)&=0,\quad q_0(t)=\frac{P(X,tX^4)}{X^{2k}},\\
q_{n+1}(t)&=\frac{1}{168}(592t-5)(2k+4n)q_n(t)-(16t^2-t)q_n'(t)\\
          &\qquad\qquad\qquad-\frac{1}{7056}n(n+k-1)(6400t^2+15584t+25)q_{n-1}(t)\qquad (n\geq 0).
\end{align*}
\end{proposition}
\begin{proof}
Let $\phi=\tfrac{1}{12}E_2-\tfrac{1}{42}(\Theta^4+16F_2)$, whence 
\[\psi=D\phi-\phi^2=-\frac 1{7056}\left(25\Theta^8+15584\Theta^4F_2+6400F_2^2\right).\] 
Then $\phi$ is a quasimodular form of weight $2$ for $\Gamma_0(4)$ and $\phi^*=\tfrac{1}{12}E_2^*-\tfrac{1}{42}(\Theta^4+16F_2)=\phi-\frac{1}{4\pi y}$ transforms like a modular form. Since $E_2^*(\z_7)=\tfrac{3}{\sqrt{7}}\Omega_{-7}^2$ (cf. \cite[Table on p. 87]{Zagier}), one sees easily by comparing to the values given in \eqref{eqvaluesz7} that $\phi^*(\z_7)=0$, wherefore it follows, as in the proof of \Cref{propTaylori}, that $\partial^n f(\z_7)=\vartheta_\phi^{[n]} f(\z_7)$. The action of $\vartheta_\phi^{[n]}$ on the polynomial algebra $\C[\Theta,F_2]$ is determined by
\begin{align*}
\vartheta_\phi \Theta&= -\frac{1}{168}\left(5\Theta^5-592\Theta F_2\right),\\
\vartheta_\phi F_2&= \frac{1}{42}\left(37\Theta^4 F_2-80F_2^2\right),
\end{align*}
as one can easily verify. The proof now follows the exact same lines as that of \Cref{propTaylori}, so we leave the rest to the reader.
\end{proof}
\begin{example}
We apply \Cref{propTaylorz7} to the Cohen-Eisenstein series 
\begin{align*}
\calH_{5/2}(\tau)&=\frac{1}{120}\left(\Theta^5(\tau)-20\Theta(\tau) F_2(\tau)\right)\\
                 &=\frac{1}{120}\left(1 - 10q - 70q^4 - 48q^5 - 120q^8 - 250q^9 - 240q^{12} - 240q^{13} + O(q^{16})\right)
                 \end{align*}
of weight $5/2$. Choosing $\Phi=\frac{\sqrt{7}}{2}\pi\Omega_{-7}^2=\frac{(\Gamma(1/7)\Gamma(2/7)\Gamma(4/7))^2}{32\pi^3}$ and setting $\eps=8-3\sqrt{7}$ the fundamental unit of the field $\Q(\sqrt 7)$ we find that
\[(1-w)^{-5/2}\calH_{5/2}\left(\frac{\z_7-\overline{\z_7}w}{1-w}\right)=\frac{\Theta(\z_7)^5}{480\eps}\sum_{n=0}^\infty \frac{d(n)}{n!}(\Phi w)^n,\]
with the first few of the numbers $d(n)$ being given by
\begin{align*}
d(n)=&-3\sqrt{7} + 72,\ -60\sqrt{7} - 265,\ 1105\sqrt{7} + 1160, -6300\sqrt{7} - 30705,\\
& 130485\sqrt{7} + 366600,\ -2715900\sqrt{7} - 5323465,\ 38437065\sqrt{7} + 146660040, \\
& -1220829660\sqrt{7} - 2376737265,\ 24402981165\sqrt{7} + 78627988680...
\end{align*}
The $d(n)$ are normalized so that they are integers in $\Q(\sqrt{7})$ which one can verify again by an induction analogous to the one employed in \Cref{exTaylori}. Note that this is not possible if we normalize so that the leading coefficient is $1$. The factored norms of these numbers are given by
\begin{align*}
\Nm(d(n))=&3^2\cdot 569,\ 5^2\cdot 1801,\ -3^3\cdot 5^2\cdot 47\cdot 227,\ 3^2\cdot 5^2\cdot 193\cdot 15313, \\
        &3^3\cdot 5^2\cdot  22535131,\ -5^2\cdot 7\cdot 401\cdot 331934593,\ 3^4\cdot 5^2\cdot 5514721764001,\\
        &-3^2\cdot 5^2\cdot 7\cdot 2797\cdot 1085992448669,\ 3^4\cdot 5^2\cdot 139\cdot 7154532998265547...
\end{align*}
Note that the factor $569$ in the norm of $d(0)$ also occurs in the norm of the singular modulus
\[\alpha:=\frac{\calH_{5/2}(\z_7)}{\Theta^5(\z_7)}=\frac{1065-400\sqrt{7}}{800},\]
which equals $\Nm(\alpha)=2^{-10}\cdot 5^{-2}\cdot 569$. 
For practical reasons, we look at the norms of the numbers $d(n)$ modulo $11$ and, computing the first 1000 of them, we find that
$\Nm(d(n))\equiv 3\Nm(d(n+110)\pmod{11}$ for $n\geq 3$. 
\end{example}

\end{document}